\documentclass[a4paper,11pt]{article}
\usepackage[latin1]{inputenc}
\usepackage[english]{babel}
\usepackage{amsmath}
\usepackage{amsfonts}
\usepackage{amssymb}
\usepackage{epsfig}
\usepackage{amsopn}
\usepackage{amsthm}
\usepackage{color}
\usepackage{graphicx}
\usepackage{subfigure}
\usepackage{enumerate}
\newtheorem{theorem}{Theorem}[section]

\newtheorem*{theorem*}{Theorem}
\newtheorem*{lemma*}{Lemma}
\newtheorem*{remark*}{Remark}
\newtheorem*{definition*}{Definition}
\newtheorem*{proposition*}{Proposition}
\newtheorem*{corollary*}{Corollary}
\numberwithin{equation}{section}
\newcommand{\real}{\mathbb{R}}

\def\qed{\,\unskip\kern 6pt \penalty 500
\raise -2pt\hbox{\vrule \vbox to8pt{\hrule width 6pt
\vfill\hrule}\vrule}\par}
\definecolor{darkblue}{rgb}{0.05, .05, .65}
\definecolor{darkgreen}{rgb}{0.1, .65, .1}
\definecolor{darkred}{rgb}{0.8,0,0}
\newcommand{\beqn}{\begin{equation}}
\newcommand{\eeqn}{\end{equation}}
\newcommand{\bear}{\begin{eqnarray}}
\newcommand{\eear}{\end{eqnarray}}
\newcommand{\bean}{\begin{eqnarray*}}
\newcommand{\eean}{\end{eqnarray*}}
\begin{document}

\title{\huge \bf A critical non-homogeneous heat equation with weighted source}

\author{
\Large Razvan Gabriel Iagar\,\footnote{Departamento de Matem\'{a}tica
Aplicada, Ciencia e Ingenieria de los Materiales y Tecnologia
Electr\'onica, Universidad Rey Juan Carlos, M\'{o}stoles,
28933, Madrid, Spain, \textit{e-mail:} razvan.iagar@urjc.es},\\
[4pt] \Large Ariel S\'{a}nchez,\footnote{Departamento de Matem\'{a}tica
Aplicada, Ciencia e Ingenieria de los Materiales y Tecnologia
Electr\'onica, Universidad Rey Juan Carlos, M\'{o}stoles,
28933, Madrid, Spain, \textit{e-mail:} ariel.sanchez@urjc.es}\\
[4pt] }
\date{}
\maketitle

\begin{abstract}
Some qualitative properties of radially symmetric solutions to the non-homogeneous heat equation with critical density and weighted source
$$
|x|^{-2}\partial_tu=\Delta u+|x|^{\sigma}u^p, \quad (x,t)\in\real^N\times(0,T),
$$
are obtained, in the range of exponents $p>1$, $\sigma\ge-2$. More precisely, we establish conditions fulfilled by the initial data in order for the solutions to either blow-up in finite time or decay to zero as $t\to\infty$ and, in the latter case, we also deduce decay rates and large time behavior. In the limiting case $\sigma=-2$ we prove the existence of non-trivial, non-negative solutions, in stark contrast to the homogeneous case. A transformation to a generalized Fisher-KPP equation is derived and employed in order to deduce these properties.
\end{abstract}

\

\noindent {\bf Mathematics Subject Classification 2020:} 35A22, 35B33, 35B36, 35B44, 35K57, 35K67.

\smallskip

\noindent {\bf Keywords and phrases:} non-homogeneous heat equation, critical density, reaction-diffusion equations, critical exponents, decay rate, finite time blow-up.

\section{Introduction}

The goal of this short note is to establish some properties of solutions to the following non-homogeneous heat equation with a weighted source
\begin{equation}\label{eq1}
|x|^{-2}\partial_tu=\Delta u+|x|^{\sigma}u^p, \quad (x,t)\in\real^N\times(0,T), \quad T\in(0,\infty],
\end{equation}
in the range of exponents $\sigma\ge-2$, $p>1$. More precisely, we consider radially symmetric solutions to \eqref{eq1} and we aim at establishing conditions on the exponents $p$, $\sigma$ and on the (radially symmetric) initial condition
\begin{equation}\label{ic}
u(x,0)=u_0(|x|)\in C(\real^N)\cap L^{\infty}(\real^N), \quad u_0\geq0, \quad u_0\not\equiv0,
\end{equation}
implying either finite time blow-up of solutions to Eq. \eqref{eq1} or global existence and decay as $t\to\infty$, and in the latter case giving a decay rate and (under suitable conditions) the large time behavior as $t\to\infty$. In order to ease the notation, we denote the radial variable by $r=|x|$ and we recall here for the readers' convenience that, in this variable, Eq. \eqref{eq1} writes
\begin{equation}\label{eq2}
r^{-2}u_t=u_{rr}+\frac{N-1}{r}u_r+r^{\sigma}u^p,
\end{equation}
where, as usual, the subscripts indicate derivatives.

Eq. \eqref{eq1} features a competition between a non-homogeneous heat equation with a weight (usually referred as density function by physical reasons) $\varrho(x)=|x|^{-2}$ and a source term also weighted with a power $|x|^{\sigma}$. Equations in the more general form 
\begin{equation}\label{nhpme}
\varrho(x)u_t=\Delta u^m+q(x)u^p, \quad m\geq1, \quad p>1,
\end{equation}
for suitable functions $\varrho(x)$ (usually called \emph{density function}) and $q(x)$ have been proposed in several models in radial transport in confined plasma \cite{KR81, KR82}, in the theory of combustion with a power-law temperature depending on the source \cite{KKMS80, KK04}, or (without the reaction term) in a kinetic model that describes the evolution of the probability density of the number of firms in a society \cite{T16}, just to give some examples. Such models justified a development of the mathematical theory of equations in the form \eqref{nhpme}, with or without a reaction term, and such equations are usually referred in literature under the name of \emph{non-homogeneous heat equation} (if $m=1$) or \emph{non-homogeneous porous medium equation} (if $m>1$).

It has been thus noticed that densities $\varrho(x)$ either being exactly equal to $|x|^{-2}$ or behaving as $|x|^{-2}$ as $|x|\to\infty$ are \emph{critical for the dynamic properties} of \eqref{nhpme}. Indeed, restricting ourselves to our case of interest $m=1$, a number of works have addressed the mathematical analysis of solutions to Eq. \eqref{nhpme}, but the case $\varrho(x)\sim |x|^{-2}$ has been avoided therein, see for example \cite{dPRS13, LX14, ZW08}, where densities behaving like $|x|^{-\sigma}$ are considered for either $\sigma\in(0,2)$ or $\sigma>2$. A similar criticality of $\varrho(x)=|x|^{-2}$ has been observed also in the porous medium case $m>1$ in \cite{IS14, KRV10, KKMS80, KK04, MT07, MP20, MP21, MP21b, RV06}, to quote but a few references dealing with the mathematical study of \eqref{nhpme}.

Regarding the equation \eqref{nhpme} with $\varrho(x)=|x|^{-2}$, the authors established well-posedness and large time behavior of solutions to \eqref{nhpme} with $m=1$ and $q(x)=0$ in \cite{IS13} (generalizing then the technique also to the non-homogeneous porous medium equation without source term in \cite{IS14}), these properties being deduced by exploiting a transformation mapping its radially symmetric solutions to those of a standard heat equation. Afterwards, Toscani described in \cite{T16} a kinetic model leading to a rather similar equation and employing an alternative transformation in its study. More recently, the authors considered in \cite{IS23} a general class of equations with two weights, namely
\begin{equation}\label{eq1.gen}
|x|^{\sigma_1}\partial_tu=\Delta u^m+|x|^{\sigma_2}u^p
\end{equation}
and established a number of qualitative properties of their solutions (such as Fujita-type exponents, second critical exponent limiting between finite time blow-up and global existence, classification of self-similar solutions) by introducing and then exploiting a number of transformations, at the level of radially symmetric solutions, mapping solutions to Eq. \eqref{eq1.gen} to solutions to some other reaction-diffusion equations whose properties were well-established prior to this research. However, Eq. \eqref{eq1} (which is actually a limiting case for the ranges considered in \cite{IS23}) has been left out of the study in \cite{IS23} due to the fact that the main bunch of transformations therein were not applicable to it.

The present work is thus devoted to fill in this gap by establishing some qualitative properties realted to the evolution of radially symmetric solutions to Eq. \eqref{eq1}. This is done by means of a transformation detailed in Section \ref{sec.transf} mapping such solutions to Eq. \eqref{eq1} to solutions to a generalized Fisher-KPP type equation, which allows us to translate features of the latter equation to the former one by undoing the transformation. With the aid of it, we are able to give conditions on $p>1$ and on the radially symmetric initial condition $u_0$ such that the solution to the Cauchy problem \eqref{eq1}-\eqref{ic} either blows up in a finite time $T\in(0,\infty)$ or is global in time and decays to zero as $t\to\infty$; in the latter, we also give the large time behavior of such solutions.

In order to fix the notation employed throughout this short note, let us introduce here two critical exponents who will play a significant role in the forthcoming study,
\begin{equation}\label{crit.exp}
p_c(\sigma):=\frac{N+\sigma}{N-2}, \qquad p_s(\sigma):=\frac{N+2\sigma+2}{N-2}, \quad N\geq3,
\end{equation}
the second being usually referred to as the \emph{Sobolev critical exponent} in the context of the reaction-diffusion equation
$$
\partial_tu=\Delta u+|x|^{\sigma}u^p,
$$
see for example \cite{FT00, MS21}, both of them being equal by convention to $+\infty$ in dimension $N\in\{1,2\}$. In particular, the analysis in Sections \ref{sec.bu} and \ref{subsec.ltb} shows that $p_c(\sigma)$ is the \emph{Fujita-type exponent} to Eq. \eqref{eq1} in dimension $N\geq3$, while $p_s(\sigma)$ appears in relation to the behavior of a \emph{new explicit solution} being a separatrix between global existence and blow-up, see Section \ref{subsec.sep}.

Before passing to the precise statements and proofs, we stress that we do not aim at constructing a functional-analytic theory of Eq. \eqref{eq1} in this note. However, in the particular framework of radial symmetry, the well-posedness for Eq. \eqref{eq2} follows directly from the well-posedness property of the resulting Fisher-KPP type equation by undoing the transformation in Section \ref{sec.transf}. We refrain here from entering the more detailed discussion of a notion of weak solution and the well-posedness, regularity and other properties of Eq. \eqref{eq1} without the assumption of radial symmetry.

\section{The transformation}\label{sec.transf}

The following transformation, which is the main tool in our analysis, has been announced (but not employed in any form) in \cite{IS23}. We describe it below in more detail. Let $u$ be a radially symmetric solution to Eq. \eqref{eq1}, that is, a solution to Eq. \eqref{eq2}. In a first step, set
\begin{equation}\label{transf1}
w(y,t)=r^{\frac{\sigma+2}{p-1}}u(r,t), \quad y=\ln\,r, \quad r\in(0,\infty).
\end{equation}
After performing straightforward calculations, we deduce that $w(y,t)$ solves the equation
\begin{equation}\label{interm1}
w_t=w_{yy}+Kw_y-K_0w+w^p,
\end{equation}
with
\begin{equation}\label{interm2}
K_0:=\frac{\sigma+2}{p-1}\left[N-2-\frac{\sigma+2}{p-1}\right], \quad K:=N-2-\frac{2(\sigma+2)}{p-1}.
\end{equation}
Observe that, for $N\geq3$, the previous constants can be written in terms of the critical exponents \eqref{crit.exp} as
$$
K_0=\frac{(\sigma+2)(N-2)}{(p-1)^2}(p-p_c(\sigma)), \quad K=\frac{N-2}{p-1}(p-p_s(\sigma)).
$$
In a second step, we introduce the function $\Psi$ as follows:
\begin{equation}\label{transf2}
\Psi(z,t)=w(y,t), \quad z=y+Kt,
\end{equation}
with $K$ defined in \eqref{interm2}. Then the function $\Psi$ solves the following one-dimensional generalized Fisher-KPP type equation
\begin{equation}\label{Fisher}
\Psi_t=\Psi_{zz}-K_0\Psi+\Psi^p.
\end{equation}
We may thus compose the transformations \eqref{transf1} and \eqref{transf2} to map \eqref{eq2} into \eqref{Fisher}. Let us also notice here that the initial conditions for the two equations are mapped in the following way
\begin{equation}\label{transf.ic}
u_0(r)=r^{-(\sigma+2)/(p-1)}w_0(r)=e^{-(\sigma+2)z/(p-1)}\Psi_0(z),
\end{equation}
a fact that is useful when dealing with Cauchy problems. Since the well posedness of \eqref{Fisher} is a standard property in the sense of classical solutions (for the uniqueness and comparison principle, even in more general cases, see for example \cite{DK12, DJ23}), and the transformations \eqref{transf1} and \eqref{transf2} preserve the order between two solutions at a fixed point $(r,t)$, we obtain local existence and uniqueness for \eqref{eq2} by undoing the previous transformations (with solutions that remain classical except, in some cases, at $r=0$). The properties of solutions to \eqref{Fisher} needed in the forthcoming analysis have been deduced by the authors (in a framework of a more general study) in \cite{IS24}. We are thus in a position to translate these features into properties of solutions to Eq. \eqref{eq2} by undoing the previous changes of variable \eqref{transf1}-\eqref{transf2}.

\section{Finite time blow-up for $1<p<p_c(\sigma)$ or $N\in\{1,2\}$}\label{sec.bu}

The result below, together with the one in the next Section \ref{sec.decay}, show that the exponent $p_c(\sigma)$ introduced in \eqref{crit.exp} plays the role of a Fujita-type exponent for Eq. \eqref{eq2}.
\begin{theorem}\label{th.fujita}
Let $\sigma>-2$ and either $N\in\{1,2\}$ or $N\geq3$ and $p\in(1,p_c(\sigma))$. Then any non-trivial, non-negative radially symmetric solution to Eq. \eqref{eq1} blows up in finite time.
\end{theorem}
\begin{proof}
Notice that any of the conditions $N\in\{1,2\}$ or $N\geq3$ and $1<p<p_c(\sigma)$ implies $K_0<0$, where $K_0$ is defined in \eqref{interm2}. Thus, \eqref{Fisher} features in this case a sum of reaction terms. It is rather obvious then that any solution to \eqref{Fisher} blows up in finite time, but for the sake of completeness we give some details. Let $\Psi$ be a non-negative, non-trivial solution to \eqref{Fisher}. Then, in particular $\Psi$ is a supersolution to the linear equation
\begin{equation}\label{interm3}
\underline{\Psi}_t=\underline{\Psi}_{zz}+|K_0|\underline{\psi}.
\end{equation}
Let $\underline{\Psi}$ be the solution to \eqref{interm3} with initial condition $\underline{\Psi}(z,0)=\Psi(z,0)$ and introduce
$$
\Phi(z,t)=e^{-|K_0|t}\underline{\Psi}(z,t).
$$
Then $\Phi(z,0)=\underline{\Psi}(z,0)$ and $\Phi$ is a solution to the classical one-dimensional heat equation. It then follows from standard results related to the heat equation that
$$
\underline{\Psi}(z,t)=e^{|K_0|t}\Phi(z,t)
$$
has at least an exponential time growth on compact subsets. By comparison, the same holds true for $\Psi$. But, on the other hand, $\Psi$ is also a supersolution to
\begin{equation}\label{interm4}
\Psi_t=\Psi_{zz}+\Psi^p,
\end{equation}
and, since it has an exponential time growth, one can readily see that the energy
$$
E[\Psi](t)=\frac{1}{2}\int_{\real}|\nabla\Psi(z,t)|^2\,dz-\frac{1}{p+1}\int_{\real}\Psi^{p+1}(z,t)\,dz
$$
becomes negative for some $t_0>0$ sufficiently large. Considering this $t_0$ as initial time for comparison, finite time blow-up of $\Psi$ is then ensured by \cite[Theorem 17.6]{QS} and the comparison principle. We end the proof by undoing the transformations \eqref{transf1}-\eqref{transf2}.
\end{proof}

\section{Transition from decay to blow-up for $p>p_c(\sigma)$}\label{sec.decay}

In the range $p>p_c(\sigma)$, we observe that $K_0>0$ and thus \eqref{Fisher} is a generalized Fisher-KPP type equation, featuring a competition between a (linear) absorption term and a source term. As established in \cite{IS24}, this fact leads to more interesting situations than in the previous section. Throughout this section, we fix $N\geq3$ and $p>p_c(\sigma)$, the case $p=p_c(\sigma)$ being considered in the next section.

\subsection{Decay and large time behavior}\label{subsec.ltb}

Let us first remark that, for $p>p_c(\sigma)$, Eq. \eqref{eq1} admits a stationary solution (with a singularity at $x=0$),
\begin{equation}\label{stat.sol}
S(r):=K_0^{1/(p-1)}r^{-(\sigma+2)/(p-1)}.
\end{equation}
In this section we establish the decay rate as $t\to\infty$ and the large time behavior of solutions whose initial condition lies below the singular stationary solution \eqref{stat.sol}. Furthermore, we establish at least a class of solutions that remain bounded globally in time, completing the characterization of $p_c(\sigma)$ as the Fujita-type exponent for Eq. \eqref{eq1}.
\begin{theorem}\label{th.decay}
Let $N\geq3$, $\sigma>-2$ and $p>p_c(\sigma)$. Let $u_0$ be an initial condition as in \eqref{ic} such that $u_0(r)<S(r)$ for any $r\in(0,\infty)$. Then
\begin{enumerate}
\item There exists $C=C(p,\sigma,u_0)>0$ such that the solution $u$ to Eq. \eqref{eq1} with initial condition $u_0$ satisfies
\begin{equation}\label{decay}
u(r,t)\leq Cr^{-(\sigma+2)/(p-1)}e^{-K_0t}, \quad (r,t)\in(0,\infty)\times(0,\infty).
\end{equation}
\item If, furthermore, the following weighted integral is finite
\begin{equation}\label{int.weight}
M(u_0):=\int_0^{\infty}r^{(\sigma+2)/(p-1)-1}u_0(r)\,dr<\infty,
\end{equation}
then the large time behavior of $u$ as $t\to\infty$ is given by
\begin{equation}\label{asympt.beh}
\lim\limits_{t\to\infty}t^{1/2}\left[e^{K_0t}r^{(\sigma+2)/(p-1)}u(r,t)-G(r+Kt,t)\right]=0,
\end{equation}
with uniform convergence over compact sets in $(0,\infty)$, where $K_0$, $K$ are defined in \eqref{interm2} and
$$
G(\zeta,t)=\frac{M(u_0)}{\sqrt{4\pi t}}e^{-\zeta^2/4t},
$$
is the Gaussian kernel.
\item If, moreover, ${\rm supp}\,u_0\subset\real^N\setminus\{0\}$ is a compact set, then $u(0,t)=0$ for any $t\in(0,\infty)$ and in particular $u(t)\in L^{\infty}(\real^N)$ for any $t>0$.
\end{enumerate}
\end{theorem}
\begin{proof}
Let $u_0$ be as in the statement. We infer from \eqref{transf.ic} that, by applying the transformation \eqref{transf1}-\eqref{transf2}, the radially symmetric solution to Eq. \eqref{eq1} with initial condition $u_0$ is mapped into the solution to \eqref{Fisher} with an initial condition $\Psi_0$ satisfying
\begin{equation}\label{interm5}
\|\Psi_0\|_{\infty}<K_0^{1/(p-1)}.
\end{equation}
We derive from \eqref{interm5} and \cite[Theorem 1]{IS24} that the solution $\Psi$ to \eqref{Fisher} with initial condition $\Psi_0$ is global, decays to zero as $t\to\infty$ and, more precisely, there is $C>0$ (depending on $\|\Psi_0\|_{\infty}$) such that
$$
\Psi(z,t)\leq Ce^{-K_0t}, \quad (z,t)\in\real\times(0,\infty).
$$
The decay estimate \eqref{decay} follows then from the previous inequality by undoing \eqref{transf1}-\eqref{transf2}. Assume next that \eqref{int.weight} holds true. We infer from \eqref{transf1}-\eqref{transf2} that this condition entails
$$
\int_{\real}\Psi_0(z)\,dz=\int_{\real}w_0(y)\,dy=\int_0^{\infty}\frac{w_0(r)}{r}\,dr=\int_0^{\infty}r^{(\sigma+2)/(p-1)-1}u_0(r)\,dr<\infty,
$$
thus $\Psi_0\in L^1(\real)$ and $M(u_0)=\|\Psi_0\|_1$. Then, \cite[Theorem 1]{IS24} ensures that
\begin{equation}\label{interm6}
\lim\limits_{t\to\infty}t^{1/2}\|e^{K_0t}\Psi(t)-G(t)\|_{\infty}=0,
\end{equation}
with
$$
G(z,t)=\frac{\|\Psi_0\|_{1}}{\sqrt{4\pi t}}e^{-z^2/4t}.
$$
We then deduce the large time behavior \eqref{asympt.beh} (with uniform convergence) from \eqref{interm6} by undoing the transformation \eqref{transf1}-\eqref{transf2}. Finally, assume that the initial condition $u_0$ is compactly supported in a set contained in $\real^N\setminus\{0\}$, or, in terms of radial variables, in a compact interval of $(0,\infty)$. We then deduce from \eqref{transf.ic} that the initial condition $\Psi_0$ to \eqref{Fisher} is compactly supported in $(-\infty,\infty)$. Let $\tilde{\Psi}$ be the solution to the Cauchy problem
\begin{equation}\label{interm9}
\begin{split}
&\tilde{\Psi}_{t}-\tilde{\Psi}_{zz}=0, \quad (z,t)\in\real\times(0,\infty),\\
&\tilde{\Psi}(z,0)=\Psi_0(z), \quad z\in\real.
\end{split}
\end{equation}
Since $\|\Psi_0\|_{\infty}<K_0^{1/(p-1)}$, it readily follows by the comparison principle that
\begin{equation}\label{interm10}
\|\tilde{\Psi}(t)\|_{\infty}<K_0^{1/(p-1)}, \quad t\in(0,\infty).
\end{equation}
Moreover, $\tilde{\Psi}$ is a supersolution to \eqref{Fisher}. Indeed, taking into account \eqref{interm9} and \eqref{interm10}, we compute
$$
\tilde{\Psi}_t(z,t)-\tilde{\Psi}_{zz}(z,t)+K_0\tilde{\Psi}(z,t)-\tilde{\Psi}^p(z,t)=\tilde{\Psi}(z,t)(K_0-\tilde{\Psi}^{p-1}(z,t))>0,
$$
for any $(z,t)\in\real\times(0,\infty)$. It thus follows from the comparison principle applied to \eqref{Fisher} that $\Psi(z,t)\leq\tilde{\Psi}(z,t)$ for any $(z,t)\in\real\times(0,\infty)$. Moreover, by standard results for the heat equation, there is a mapping $t\mapsto z(t)$ and a positive constant $C(t)$ such that we have
$$
\Psi(z,t)\leq\tilde{\Psi}(z,t)\leq C(t)e^{-z^2/4t}, \quad t>0, \quad z\in(-\infty,z(t)),
$$
whence, by undoing the transformations \eqref{transf1}-\eqref{transf2}, we obtain that $u(0,t)=0$ for any $t>0$ and that the solution $u$ remains bounded in a neighborhood of $r=0$. This fact, together with \eqref{decay}, ensures that $u(t)\in L^{\infty}(\real^N)$ for any $t\in(0,\infty)$, completing the proof.
\end{proof}

\medskip 

\noindent \textbf{Discussion.} The previous statement does not take into account the behavior of the function $t\mapsto u(r,t)$ at $r=0$ when we start with data such that $u_0(0)>0$. Indeed, a priori, the result of Theorem \ref{th.decay} stays true even when $u(r,t)$ might blow-up in the sense of developing a vertical asymptote at $r=0$ with a weaker singularity than $r^{-(\sigma+2)/(p-1)}$, while remaining bounded (and decaying in time as in the statement) at any point $x\in\real^N$ with $r=|x|>0$. This property is related, by the transformation \eqref{transf1}-\eqref{transf2}, with the question of whether, if $\Psi$ is a solution to \eqref{Fisher} such that
\begin{equation}\label{interm8}
\Psi_0(z)\sim Ce^{(\sigma+2)z/(p-1)}, \quad {\rm as} \ z\to-\infty,
\end{equation}
then its tail is preserved at later times, that is,  
$$
\Psi(z,t)\leq\overline{C}e^{(\sigma+2)z/(p-1)}, \quad z\in(-\infty,z(t)), \quad t>0.
$$
Up to our knowledge, there is no reference proving such a result for the precise decay \eqref{interm8}, but it is established in \cite{He99} that tails of the form $\Psi(z)\sim A|z|^{-\alpha}$ as $|z|\to\infty$ are preserved during the evolution for solutions to the reaction-diffusion equation
$$
\Psi_t=\Psi_{zz}+\Psi^p,
$$
and by comparison, the upper bound for the tail remains in force for \eqref{Fisher}. It is very likely (as suggested by the author of \cite{He99}, who claims that the techniques therein apply for more general decays) that the same property holds true for initial conditions decaying with a precise exponential tail as $z\to-\infty$ (such as, for example, \eqref{interm8}). However, due to the very technical and lengthy character of the proofs in \cite{He99}, we refrain from developing this subject in the present note.

\subsection{A new explicit solution as a separatrix between decay and blow-up}\label{subsec.sep}

In this section, we introduce a new explicit radially symmetric solution to Eq. \eqref{eq1} and show that it is a separatrix between decay and blow-up, in some conditions. As we shall see, the Sobolev critical exponent $p_s(\sigma)$ defined in \eqref{crit.exp} plays a significant role. We start from the following stationary solution to \eqref{Fisher} obtained in \cite[Section 2.1]{IS24}
\begin{equation}\label{stat.fisher}
\overline{\Psi}(z)=\left[\frac{K_0(p+1)}{2}\right]^{1/(p-1)}\left[1-\tanh^2\left(\frac{(p-1)\sqrt{K_0}}{2}z\right)\right]^{1/(p-1)},
\end{equation}
which, by undoing first the transformation \eqref{transf2} leads to
\begin{equation*}
W(y,t)=\left[\frac{K_0(p+1)}{2}\right]^{1/(p-1)}\left[1-\tanh^2\left(\frac{(p-1)\sqrt{K_0}}{2}(y+Kt)\right)\right]^{1/(p-1)}
\end{equation*}
and finally, by undoing also \eqref{transf1}, gives the following explicit radially symmetric solution to Eq. \eqref{eq1}
\begin{equation}\label{expl.sol}
U(r,t)=\left[\frac{K_0(p+1)}{2}r^{-(\sigma+2)}\right]^{1/(p-1)}\left[1-\tanh^2\left(\frac{(p-1)\sqrt{K_0}}{2}(\ln\,r+Kt)\right)\right]^{1/(p-1)}.
\end{equation}
Observe that $U$ is an \emph{eternal solution} to Eq. \eqref{eq1}, in the sense that it is well defined for $t\in(-\infty,\infty)$. By employing the identity
$$
1-\tanh^2(\theta)=\frac{4}{2+e^{2\theta}+e^{-2\theta}}, \quad \theta=\frac{(p-1)\sqrt{K_0}}{2}(\ln\,r+Kt)
$$
we obtain from \eqref{expl.sol} that, for any $t>0$,
\begin{equation}\label{beh.stat}
\begin{split}
&U(r,t)\sim[2K_0(p+1)]^{1/(p-1)}r^{\sqrt{K_0}-(\sigma+2)/(p-1)}e^{\sqrt{K_0}Kt}, \quad {\rm as} \ r\to0,\\
&U(r,t)\sim[2K_0(p+1)]^{1/(p-1)}r^{-\sqrt{K_0}-(\sigma+2)/(p-1)}e^{-\sqrt{K_0}Kt}, \quad {\rm as} \ r\to\infty,
\end{split}
\end{equation}
where, as usual, the equivalence symbol $\sim$ means that the quotient between the two sides tends to one in the specified limit, and we recall that $K$ and $K_0$ are defined in \eqref{interm2}. We observe that $U$ always decays as $r\to\infty$, but the local behavior at $r=0$ and the variation with respect to time are more interesting. Indeed, we notice on the one hand that
$$
K_0-\left(\frac{\sigma+2}{p-1}\right)^2=\frac{(\sigma+2)(N-2)}{(p-1)^2}[p-p_s(\sigma)],
$$
thus $p=p_s(\sigma)$ is critical with respect to the behavior at $r=0$. On the other hand, studying the variation with respect to time of the expression
$$
h(t):=\frac{4}{2+e^{(p-1)\sqrt{K_0}(\ln\,r+Kt)}+e^{-(p-1)\sqrt{K_0}(\ln\,r+Kt)}},
$$
we observe that it behaves in different ways in dynamic inner sets $\{0<r<e^{-Kt}\}$ and in dynamic outer sets $\{r>e^{-Kt}\}$, for $t>0$. Gathering the previous analysis, we have:

$\bullet$ if $p\in(p_c(\sigma),p_s(\sigma))$, then $U$ is a singular solution, presenting a vertical asymptote at $r=0$. Moreover, according to \eqref{beh.stat} and the negativity of $K$ in this range, the solution $U$ decays with respect to time in the inner set $\mathcal{I}$ and increases with respect to time in the outer set $\mathcal{O}$ defined below:
\begin{equation}\label{inner.outer}
\mathcal{I}:=\{(r,t)\in(0,\infty)^2: 0<r<e^{-Kt}\}, \quad \mathcal{O}:=\{(r,t)\in(0,\infty)^2: e^{-Kt}<r\}.
\end{equation}

$\bullet$ if $p=p_s(\sigma)$, then we also observe that $K=0$ and thus $U$ is a stationary solution to Eq. \eqref{eq1} with $U(0,t)=U(0)=[2K_0(p+1)]^{1/(p-1)}$.

$\bullet$ if $p>p_s(\sigma)$, then we have $U(0,t)=0$ for any $t>0$. Moreover, according to \eqref{beh.stat} and the positivity of $K$ in this range, the solution $U$ increases with respect to time in the inner set $\mathcal{I}$ and decreases with respect to time in the outer set $\mathcal{O}$ introduced in \eqref{inner.outer}.

In order to better illustrate the previous behavior in inner and outer sets, for the readers' convenience we plot the solution $U$ at different times in Figure \ref{fig1}.

\begin{figure}[ht!]
  \begin{center}
  \subfigure[$p<p_s(\sigma)$]{\includegraphics[width=7.5cm,height=6cm]{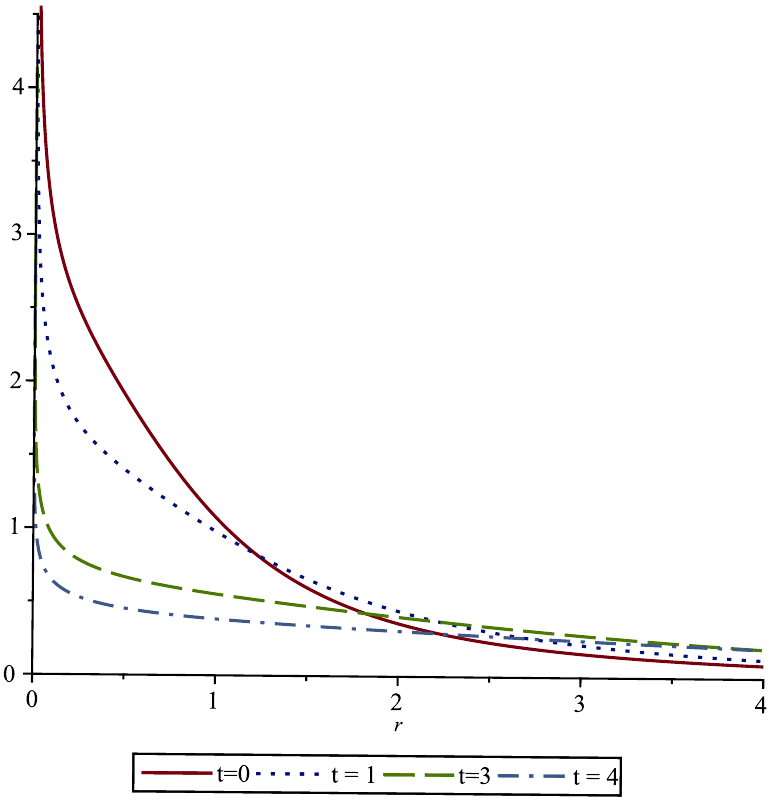}}
  \subfigure[$p>p_s(\sigma)$]{\includegraphics[width=7.5cm,height=6cm]{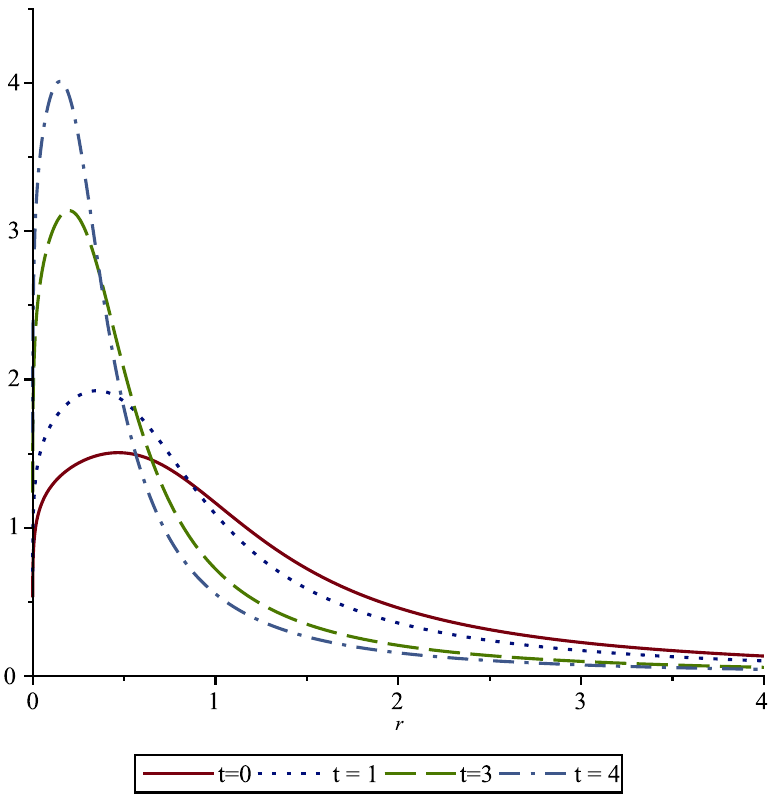}}
  \end{center}
  \caption{The solution $U$ plotted at different times. Experiments for $\sigma=1$, $N=4$, $p=3.5$, respectively $p=4.5$, where $p_s(\sigma)=4$.}\label{fig1}
\end{figure}

\medskip

The explicit solution $U$ plays the role of a separatrix between finite time blow-up and time decay. This is made precise in the following statement, where, as usual, $U_0(r)=U(r,0)$, $r\in[0,\infty)$.
\begin{theorem}\label{th.sep}
Let $N\geq3$, $p>p_c(\sigma)$ and $u_0$ be as in \eqref{ic}.

(a) If
\begin{equation}\label{inf.upper}
\inf\limits_{r\in(0,\infty)}\frac{u_0(r)}{U_0(r)}=\kappa_0>1,
\end{equation}
then the radially symmetric solution $u$ to Eq. \eqref{eq1} with initial condition $u_0$ blows up in finite time.

(b) If
\begin{equation}\label{sup.lower}
\sup\limits_{r\in(0,\infty)}\frac{u_0(r)}{U_0(r)}=\kappa^0<1,
\end{equation}
then the radially symmetric solution $u$ to Eq. \eqref{eq1} with initial condition $u_0$ decays to zero as $t\to\infty$ on compact subsets in $\real^N\setminus\{0\}$ and behaves as indicated in Theorem \ref{th.decay}.
\end{theorem}
\begin{proof}
Let $\Psi$ be the solution to \eqref{Fisher} obtained from $u$ through the transformation \eqref{transf1}-\eqref{transf2} and $\Psi_0(z)=\Psi(z,0)$ its initial condition. We deduce from \eqref{transf.ic} that the conditions \eqref{inf.upper}, respectively \eqref{sup.lower}, become
$$
\inf\limits_{z\in\real}\frac{\Psi_0(z)}{\overline{\Psi}(z)}=\kappa_0>1, \quad {\rm respectively}, \quad \sup\limits_{z\in\real}\frac{\Psi_0(z)}{\overline{\Psi}(z)}=\kappa^0<1,
$$
where we recall that $\overline{\Psi}$ is the explicit stationary solution to \eqref{Fisher} introduced in \eqref{stat.fisher}. An application of \cite[Theorem 3.2]{IS24}, which states that $\overline{\Psi}$ is a separatrix between finite time blow-up on the one hand and global existence and decay as $t\to\infty$ on the other hand, completes the proof.
\end{proof}

\noindent \textbf{Remark.} Because of the singularity of $U$ at $r=0$ and the boundedness assumed for $u_0$ in \eqref{ic}, Part (a) in Theorem \ref{th.sep} cannot apply in the range $p\in(p_c(\sigma),p_s(\sigma))$.

\section{The special cases $p=p_c(\sigma)$ and $\sigma=-2$}\label{subsec.pc}

These two critical cases are connected by the common fact that, either if $\sigma>-2$ and $p=p_c(\sigma)$ or if $\sigma=-2$ and $p>1$ arbitrary, we have $K_0=0$ and thus \eqref{Fisher} becomes the standard reaction-diffusion equation 
\begin{equation}\label{eq.rd}
\Psi_t=\Psi_{zz}+\Psi^{p},
\end{equation} 
which is well understood nowadays (see, for example, the monograph \cite{QS} and references therein). The first case of this section, dealing with the exponent $p=p_c(\sigma)$, is simpler and is to be considered more as a discussion, while we extend more the analysis of the second case, $\sigma=-2$.

\bigskip

\noindent \textbf{Case 1:} $\mathbf{p=p_c(\sigma)}$. Observe that, in this case, an initial condition such that $u_0(r)>0$ for $r\in(0,\delta)$ is mapped by \eqref{transf.ic} into an initial condition $\Psi_0$ for \eqref{eq.rd} with an exponential tail as $z\to-\infty$. Moreover, recalling the Fujita exponent $p_F=3$ of \eqref{eq.rd}, we observe that
$$
p_c(\sigma)-p_F=\frac{N+\sigma}{N-2}-3=\frac{\sigma-2(N-3)}{N-2}.
$$
We thus infer from the theory in \cite[Chapter 18]{QS} that, if $\sigma\leq2(N-3)$, then $p_c(\sigma)\leq p_F$ and thus any non-trivial and non-negative solution to \eqref{eq.rd} blows-up in finite time, a property that is inherited by any non-trivial radially symmetric solution to Eq. \eqref{eq1} through the transformation \eqref{transf1}-\eqref{transf2}. On the contrary, if $\sigma>2(N-3)$, then $p_c(\sigma)>p_F$ and thus there are both solutions that blow-up in finite time and that exist globally. For example, \cite[Theorem 20.1]{QS} ensures that small initial conditions $\Psi_0$ will produce global solutions to \eqref{eq.rd} remaining below a Gaussian profile, whose decay rate as $z\to-\infty$ is faster than the one given by \eqref{interm8} and by undoing the transformation \eqref{transf1}-\eqref{transf2} we obtain thus global solutions to Eq. \eqref{eq1}. Moreover, the other results in \cite[Chapter 20]{QS} can be readily translated into properties of radially symmetric solutions to Eq. \eqref{eq1}.

\medskip

\noindent \textbf{Remark.} Let us stress here that, contrary to the standard reaction-diffusion equations, the Fujita-type exponent $p=p_c(\sigma)$ for Eq. \eqref{eq1} is not always included in the range where any solution blows up in finite time, as explained above.

\bigskip 

\noindent \textbf{Case 2:} $\mathbf{\sigma=-2}$ \textbf{and} $\mathbf{p>1}$. We observe that, in this case, \eqref{transf1} reduces to setting $r=e^y$, $y\in\real$. Then, according to the behavior of the initial condition at $r=0$, we have the following result, slightly reminding of the properties of the nonhomogeneous heat equation in \cite{IS13} (where $B(u_0)$ designs the blow-up set of the solution with initial condition $u_0$):
\begin{theorem}\label{th.limit}
Let $u_0$ be as in \eqref{ic}. 

(a) If $u_0(0)=A>0$, then the solution $u$ to Eq. \eqref{eq1} with initial condition $u_0$ blows up in finite time. Moreover, if $u_0$ is decreasing with respect to $r$, then $u$ blows up \emph{only} at $r=0$, that is, $B(u_0)=\{0\}$.

(b) Let $q_c=(p-1)/2$ and pick $q\geq\min\{1,q_c\}$. Then, if $u_0$ satisfies
\begin{equation}\label{interm7}
\int_{0}^{\infty}\frac{u_0^q(r)}{r}\,dr<\infty
\end{equation}
for some $q\geq q_c$, there exists a unique, classical, radially symmetric solution $u$ to Eq. \eqref{eq1} defined on a maximal time $t\in(0,T(u_0))$ (where $T(u_0)\in(0,\infty]$) which also satisfies
\begin{equation}\label{interm7bis}
\int_{0}^{\infty}\frac{u^q(r,t)}{r}\,dr<\infty, \quad t\in(0,T(u_0)).
\end{equation}
Moreover, if $u_0$ satisfies \eqref{interm7} and is non-decreasing for $r\in(0,1)$, then $u(0,t)=0$ for any $t\in(0,T(u_0))$.
\end{theorem}
\begin{proof}
(a) By applying \eqref{transf.ic}, the initial condition $u_0$ is mapped into an initial condition $\Psi_0$ to \eqref{eq.rd} such that $\lim\limits_{z\to-\infty}\Psi_0(z)=A>0$. It follows then from the classical work \cite{LN92} that the solution $\Psi$ to \eqref{eq.rd} with initial condition $\Psi_0$ blows up in a finite time $T\in(0,\infty)$. If, furthermore, $u_0$ is decreasing with respect to $r$, then also $\Psi_0$ will be decreasing with respect to $z\in\real$. We then readily deduce that the function $z\mapsto\Psi(z,t)$ is decreasing with respect to $z$ for any $t\in(0,T)$. Indeed, letting $w=\Psi_z$, we differentiate with respect to $z$ in \eqref{eq.rd} and derive the equation satisfied by $w$, that is,
\begin{equation*}
w_t=w_{zz}+p\Psi^{p-1}w,
\end{equation*}
a linear equation solved by $w=0$. By comparison, since $\Psi_{0,z}\leq0$, it follows that $\Psi_z(z,t)\leq0$ for any $z\in\real$, $t\in(0,T)$. We then infer from \cite[Theorem 1]{Shimojo08} that $\Psi$ blows up only at $-\infty$, whence, by undoing the transformation \eqref{transf1}-\eqref{transf2}, we conclude that $u$ blows up only at $r=0$ as $t\to T$, as claimed.

\medskip 

(b) By applying the transformation \eqref{transf.ic} with $\sigma=-2$, the initial condition $u_0$ is mapped into an initial condition $\Psi_0$ to \eqref{eq.rd}, while the condition \eqref{interm7} becomes
$$
\int_{0}^{\infty}\frac{u_0^q(r)}{r}\,dr=\int_{-\infty}^{\infty}\Psi_0^q(z)\,dz<\infty,
$$
thus $\Psi_0\in L^{\infty}(\real)\cap L^q(\real)$. Standard well-posedness results for \eqref{eq.rd} (see for example \cite[Section 15]{QS}) ensure that there is a classical solution $\Psi$ to \eqref{eq.rd}, defined on a maximal existence time $T(\Psi_0)=T(u_0)\in(0,\infty]$, such that $t\mapsto\Psi(t)$ belongs to $L^{\infty}(\real)\cap L^q(\real)$ for any $t\in(0,T(u_0))$. By undoing the transformations \eqref{transf2} and $r=e^y$, we deduce the existence of a radially symmetric solution $u$ to Eq. \eqref{eq1} defined on $(0,T(u_0))$ and satisfying \eqref{interm7bis}, as claimed. Moreover, if $u_0$ is radially non-decreasing in $(0,1)$, then the initial condition $\Psi_0$ to \eqref{eq.rd} obtained by the transformation \eqref{transf.ic} belongs to $L^q(\real)$ and is non-decreasing on $(-\infty,0)$. It then follows that $\Psi_0(z)\to0$ as $z\to-\infty$ and, by similar arguments as at the end of the proof of Part (a) above, the same property stays true for $t\in(0,T(u_0))$. By undoing the transformation \eqref{transf1}-\eqref{transf2}, we find that $u(0,t)=0$ for any $t\in(0,T)$, and the proof is complete.
\end{proof}

\textbf{Remarks.} (i) In particular, a very important aspect contained in Theorem \ref{th.limit} is the \emph{existence of non-trivial positive solutions} to Eq. \eqref{eq1} with $\sigma=-2$, a fact which is in stark contrast with the full non-existence of any non-trivial and non-negative solution to the equation
$$
\partial_tu=\Delta u+|x|^{-2}u^p, \quad p>1,  
$$
obtained in \cite[Theorem 4.7]{ACP04} (with $\gamma=0$ in the notation therein). It thus appears that the presence of the critical density function $|x|^{-2}$ in front of $u_t$ is the decisive feature that guarantees the existence of positive solutions.  

(ii) The outcome of Part (b) in Theorem \ref{th.limit} is also in strong contrast with the evolution of the standard heat equation starting from similar initial conditions $u_0$, for which the solutions become positive at every point immediately. This difference is another effect of the singular density $|x|^{-2}$ at $x=0$, a fact also noticed in \cite{IS13} in absence of a reaction term. As a particular case of part (b), we also infer that, if there is $\delta>0$ such that $u_0(r)=0$ for $r\in(0,\delta)$, then $u(0,t)=0$ for any $t\in(0,T(u_0))$, similarly to what we proved for $\sigma>-2$ and $p>p_c(\sigma)$ in Theorem \ref{th.decay}.

(iii) The case $\sigma=-2$ is of special interest in some models in the theory of combustion in \cite{KKMS80}, but with a porous medium type diffusion instead of the heat equation. Moreover, this case has not been included in the study performed in \cite{IS23}.

\section{Conclusion}

In this paper we have studied a number of properties of solutions to Eq. \eqref{eq1}, which presents as main feature a critical density function $|x|^{-2}$ pondering the time derivative. On the one hand, we have shown that the critical exponent $p_c(\sigma)$ plays the role of a Fujita-type exponent, limiting between the range of blow-up for any non-trivial solution and possible existence of global solutions to the equation. On the other hand, in the range $p>p_c(\sigma)$ we have introduced a new explicit solution which serves as a separatrix between solutions presenting finite time blow-up and solutions decaying to zero as $t\to\infty$. Moreover, we provide the large time behavior of such solutions, towards a profile related to the Gaussian kernel. Finally, we show local existence of solutions for $\sigma=-2$, in contrast with the equation with a singular potential alone which is known for non-existence of any solution. The fundamental tool employed in all this proofs is a new transformation mapping solutions to Eq. \eqref{eq1} into solutions to a generalized Fisher-KPP equation which is also itself of interest. We believe that our work contributes to improve the understanding of equations involving singular potentials and densities.

\bigskip

\noindent \textbf{Acknowledgements} R. G. I. and A. S. are partially supported by the Project PID2020-115273GB-I00 and by the Grant RED2022-134301-T funded by MCIN/AEI/10.13039/ \\ 501100011033 (Spain).

\bigskip

\noindent \textbf{Data availability} Our manuscript has no associated data.

\bigskip

\noindent \textbf{Competing interest} The authors declare that there is no competing interest.

\bibliographystyle{plain}

\end{document}